\documentclass[11pt]{article}
\usepackage{geometry}
\usepackage{parskip}
\usepackage{hyperref}
\usepackage{xcolor}
\usepackage{enumitem}
\usepackage{titlesec}
\usepackage{amsmath, amssymb, amsthm}
\usepackage{graphicx}
\usepackage{array}
\usepackage{booktabs}
\usepackage{tikz-cd}
\usepackage{mathtools}
\usepackage{pgfplots}
\pgfplotsset{compat=1.18}
\usetikzlibrary{calc, arrows.meta, positioning, backgrounds}

\definecolor{gridgreen}{RGB}{144,238,144}
\definecolor{conecyan}{RGB}{0,191,255}
\definecolor{conered}{RGB}{255,0,0}
\definecolor{linebrown}{RGB}{165,42,42}
\definecolor{textblue}{RGB}{0,0,205}
\definecolor{newgreen}{RGB}{0,128,0}

\definecolor{primary}{RGB}{0, 51, 102}
\definecolor{secondary}{RGB}{100, 100, 100}
\definecolor{linkblue}{RGB}{0, 0, 200}

\titleformat{\section}{\large\bfseries}{\thesection}{1em}{}
\titleformat{\subsection}{\bfseries}{\thesubsection}{1em}{}
\titleformat{\subsubsection}{\itshape}{\thesubsubsection}{1em}{}

\hypersetup{
	colorlinks=true,
	linkcolor=linkblue,
	urlcolor=linkblue,
	citecolor=linkblue,
}

\newtheorem{theorem}{Theorem}[section]
\newtheorem{proposition}[theorem]{Proposition}

\newtheorem{corollary}[theorem]{Corollary}
\newtheorem{definition}[theorem]{Definition}
\newtheorem{example}[theorem]{Example}
\newtheorem{remark}[theorem]{Remark}

\newtheorem*{theorem*}{Theorem}

\newcommand{\PP}{\mathbb{P}}
\newcommand{\ZZ}{\mathbb{Z}}
\newcommand{\CC}{\mathbb{C}}
\newcommand{\QQ}{\mathbb{Q}}
\newcommand{\F}{\mathbb{F}}

\begin{document}
	
	\begin{center}
		{\huge A Finiteness Theorem for Quartic K3-Fibred Calabi--Yau Threefolds in Scrolls} \\[0.8cm]
		\Large Geoffrey Mboya \\[0.3cm]
	\end{center}
	
	\begin{abstract}
		\noindent We study a restricted form of Gross's finiteness problem for algebraic minimal Calabi--Yau threefolds: those fibred by quartic K3 surfaces and realised as anticanonical hypersurfaces in stacky scrolls $\PP^1\times[\PP^3/\ZZ_n]$. Beyond the ten straight-scroll families of \cite[Table~1]{MboyaSzendroi2023}, we introduce orbifold scrolls $(\PP^1\times\PP^3)/\ZZ_n$ and apply the Reid--Shepherd-Barron--Tai criterion to determine which admit canonical anticanonical hypersurfaces. Only finitely many weight vectors arise for each $n$; we classify $n=1,2,3$ completely and compute all Hodge numbers, giving fourteen deformation families in total. Engel, Filipazzi, Greer, Mauri and Svaldi \cite{EngelFilipazziGreerMauriSvaldi2025} have since established boundedness for fibred Calabi--Yau threefolds in general, settling the existence question this restricted case exemplifies; what remains, and what we supply here, is the explicit classification: weight vectors, singularity types, and Hodge data.
		
		\medskip
	\end{abstract}
	
	
	\section{Introduction}
	
	Algebraic K3 surfaces have well-understood moduli: a countable union of $19$-dimensional components of the $20$-dimensional K\"ahler moduli space. In dimension three the situation is different: non-K\"ahler Calabi--Yau threefolds occur in infinitely many topological types, so finiteness questions require restriction to the algebraic case.
	
	\begin{definition}[\cite{Gross1994}, \S0]
		A \emph{minimal Calabi--Yau threefold} $Y$ is projective with $\QQ$-factorial terminal singularities satisfying
		\[
		K_Y = 0,\qquad \chi(\mathcal{O}_Y)=h^{1,0}(Y)=h^{2,0}(Y)=0.
		\]
		A projective threefold is an \emph{algebraic Calabi--Yau threefold} if it is birational to some minimal Calabi--Yau threefold $Y$ as above, i.e.\ if $Y$ is (a choice of) its minimal model.
	\end{definition}

	\emph{Conjecture (\cite{Gross1994}). There are only finitely many families of algebraic minimal Calabi--Yau threefolds.}

	This is open in general. Gross \cite{Gross1994} proved that algebraic Calabi--Yau threefolds admitting an elliptic fibration over a rational base are \emph{birationally bounded}, a statement weaker than deformation-finiteness, using the minimal model program, the classification of the possible base surfaces, and finiteness of the relevant Tate--Shafarevich groups. Engel, Filipazzi, Greer, Mauri and Svaldi \cite{EngelFilipazziGreerMauriSvaldi2025} have since proved boundedness for fibred Calabi--Yau threefolds in general, covering in particular threefolds fibred by K3 surfaces and settling the existence question that motivates our title. What a boundedness statement of this generality cannot supply is the explicit weight vectors, singularity types, or Hodge numbers of any specific family; that effective classification, for quartic K3-fibred threefolds in (orbifold) scrolls over $\PP^1$, is the content of this paper.
	
	We study threefolds fibred by quartic K3 surfaces $S_4\subset\PP^3$, realised as anticanonical hypersurfaces in scrolls over $\PP^1$: a more rigid setting than \cite{Gross1994,EngelFilipazziGreerMauriSvaldi2025}, since the base is always $\PP^1$ and the ambient scrolls form an explicit, classifiable list.
	
	Quartic K3 surfaces head the Iano-Fletcher--Reid list of weighted K3 hypersurfaces \cite{IanoFletcher2000}. The simplest fibration is $X_{(2,4)}\subset\PP^1\times\PP^3$; complete-intersection constructions appear in \cite{CandelasDale1988,HeCandelas1990}, with fibration structures classified in \cite{Lara2017}. Mullet \cite{Mullet2009} classified nonsingular anticanonical hypersurfaces fibred by weighted K3 quartics in weighted $\PP^3$-bundles over $\PP^1$, finding nine unweighted families. Allowing isolated singularities, the author and Szendr\H{o}i \cite{MboyaSzendroi2023} found a tenth family with three ordinary double points, and the author's thesis \cite{Mboya2022} shows these ten \emph{straight} scroll families of \S2 are complete.
	
	For $n=1$ this is the starting point. $\PP^1\times\PP^3$ also admits diagonal cyclic actions mixing base and fibre; quotienting yields \emph{orbifold} scrolls $\mathbb{F}^m_n=(\PP^1\times\PP^3)/\ZZ_n$ with quotient singularities along curves, giving quartic K3-fibred Calabi--Yau threefolds outside the weighted-scroll framework above, where only the fibre carries weights.
	
	We classify weight vectors $m=(m_1,\dots,m_4)$ and orders $n$ for which the orbifold quartic threefold has at worst canonical singularities, applying the Reid--Shepherd-Barron--Tai criterion \cite{Reid1980,Reid1985,Tai1982} directly at fixed loci. Finitely many admissible weight vectors arise for every $n$; we determine them for $n=1,2,3$ and compute all Hodge numbers.
	
	\subsection*{Acknowledgements}
	Supported by the University of Cambridge Trinity College Visiting Scholars Fellowship 2023 and Simons Foundation Award 488625. I thank Mark Gross for introducing me to this problem and for discussions and comments that helped improve this work.
	
	\section{Straight quartic scrolls}
	
	For integers $a_4\ge a_3\ge a_2\ge a_1=0$, the \emph{straight scroll} is
	\[
	\F:=\F(0,a_2,a_3,a_4)=\PP(\mathcal{O}_{\PP^1}\oplus\mathcal{O}_{\PP^1}(a_2)\oplus\mathcal{O}_{\PP^1}(a_3)\oplus\mathcal{O}_{\PP^1}(a_4)),
	\]
	a nonsingular $\PP^3$-bundle over $\PP^1$. We use \emph{straight} to distinguish $\F$ from the orbifold scrolls $\F^m_n$ of \S3; the two sets of weights play unrelated roles and are given different letters ($a$, $m$) throughout. With base coordinates $t_1,t_2$ and fibre coordinates $x_1,\dots,x_4$, the anticanonical class is $-K_F=L_{2-a_2-a_3-a_4,4}$, where $L_{p,q}$ has bidegree $(p,q)$. A general member of $|{-}K_F|$, when nonempty, is
	\[
	X=V\Bigl(\sum_{(q_1,\dots,q_4)\vdash 4}\alpha_{(q_j)}(t_1,t_2)x_1^{q_1}x_2^{q_2}x_3^{q_3}x_4^{q_4}\Bigr)\subset\F,
	\]
	fibred by quartic K3 surfaces, where $\alpha_{(q_j)}$ has degree $2+(q_2-1)a_2+(q_3-1)a_3+(q_4-1)a_4$ in $35$ quartic monomials.
	
	This is the base case ($n=1$) for the classification below.
	
	\begin{theorem}[\cite{Mullet2009,MboyaSzendroi2023,Mboya2022}]\label{thm:straight}
		Exactly ten families of quartic K3-fibred Calabi--Yau threefolds with at worst canonical (in fact isolated ordinary double point) singularities embed as general anticanonical hypersurfaces in a straight scroll $\F(0,a_2,a_3,a_4)$ over $\PP^1$, listed in Table \ref{tab:straight}.
	\end{theorem}
	
	\begin{table}[h]
		\centering
		\begin{tabular}{c l l c}
			\hline
			No. & $\F(0,a_2,a_3,a_4)$ & Singularities of general $X\in|{-}K_{\F}|$ & $\dim\mathcal{M}_{-K_F}$ \\
			\hline
			1 & $\F(0,0,0,0)$ & nonsingular & 86 \\
			2 & $\F(0,0,0,1)$ & nonsingular & 118 \\
			3 & $\F(0,0,0,2)$ & nonsingular & 83 \\
			4 & $\F(0,0,1,1)$ & nonsingular & 86 \\
			5 & $\F(0,0,1,2)$ & 3 ODPs on $\mathrm{Bs}(|{-}K_{\F}|)$ & 86 \\
			6 & $\F(0,0,2,2)$ & nonsingular & 91 \\
			7 & $\F(0,1,1,1)$ & nonsingular & 73 \\
			8 & $\F(0,1,1,2)$ & nonsingular & 86 \\
			9 & $\F(0,1,1,3)$ & nonsingular & 89 \\
			10 & $\F(0,1,1,4)$ & nonsingular & 95 \\
			\hline
		\end{tabular}
		\caption{Straight quartic scrolls carrying quartic-fibred Calabi--Yau threefolds. Families 1--4 have $-K_{\F}$ base-point-free; families 6--10 are due to \cite{Mullet2009}; family 5 is due to \cite{MboyaSzendroi2023}.}
		\label{tab:straight}
	\end{table}
	
	Family 5 admits a small projective crepant resolution $\mathrm{Bl}_{D_{34}}X\to X$ along $D_{34}=\mathrm{Bs}(|{-}K_{\F}|)\cong\PP^1\times\PP^1$. Family 8, $\F(0,1,1,2)\cong\PP(\mathcal{O}_{\PP^1}(-1)\oplus\mathcal{O}_{\PP^1}^{\oplus 2}\oplus\mathcal{O}_{\PP^1}(1))$, was studied by Gross and Ruan \cite{Ruan1996}. Theorem \ref{thm:straight} settles $n=1$; we turn to $n\ge2$.
	
	\section{Cyclic quotients of the quartic scroll}
	
	\subsection{Construction}
	
	Fix $n\ge2$ and $0=m_1\le m_2\le m_3\le m_4$. Let $1\ne\delta\in\ZZ_n$ act diagonally on $\PP^1_{[t_i]}\times\PP^3_{[x_j]}$ by
	\[
	\delta\cdot([t_i],[x_j])=([t_i],[\delta^{m_j}x_j]),
	\]
	trivially on the base, and set $\mathbb{F}^m_n:=(\PP^1\times\PP^3)/\ZZ_n$. Unlike the weighted-fibre construction cited above, here the fibres are $\PP^3$ with orbifold singularities, and the total space acquires quotient singularities along the four sections $C_j = \PP^1 \times \{ [0 : \cdots : 1 : \cdots : 0] \}$, with stabiliser $\ZZ_n$ along each, provided the weights $m_j$ are pairwise distinct modulo $n$; Remark \ref{rem:wellformed} and Example \ref{ex:smooth} treat the exceptional case.
	
	A monomial $x_1^{q_1}\cdots x_4^{q_4}$, $(q_j)\vdash4$, transforms with weight $\sum_j m_jq_j$; with $\beta\equiv\sum_j m_j\pmod n$ the multilinear monomial $x_1x_2x_3x_4$ has weight $\beta$, so a $\ZZ_n$-invariant quartic of this weight always exists, and $\Omega=dt_1\wedge dt_2\wedge dx_1\wedge\cdots\wedge dx_4/dp$ is $\delta$-invariant. Hence $X=V(p)\subset\mathbb{F}^m_n$ is anticanonical and, when $H^0(K_X)=\CC$, Calabi--Yau, provided the $\ZZ_n$-action on $\PP^1\times\PP^3$ is free in codimension $1$, so that $K_{\PP^1\times\PP^3}$ descends to $K_{\F^m_n}=-2L-4H$ (Example \ref{ex:smooth}).
	
	\begin{definition}\label{def:CY}
		An orbifold hypersurface $X\subset\F^m_n$ fibred by quartic K3 surfaces is a \emph{Calabi--Yau threefold} if $H^0(K_X)=\CC$, with $\beta\in\ZZ_n$ satisfying
		\[
		m_1q_1+\cdots+m_4q_4\equiv\beta\equiv m_1+\cdots+m_4\pmod n
		\]
		for $(q_j)\vdash4$.
	\end{definition}
	
	\begin{remark}[Divisor class group]\label{rem:clgroup}
		When the $\ZZ_n$-action is free in codimension $1$, $\mathrm{Pic}(\PP^1\times\PP^3)\cong\ZZ^2$ is fixed pointwise, giving an extension
		\[
		0\to\mathrm{Hom}(\ZZ_n,\CC^*)\to\mathrm{Pic}_{\ZZ_n}(\PP^1\times\PP^3)\to\mathrm{Pic}(\PP^1\times\PP^3)^{\ZZ_n}\to0.
		\]
		This splits: the diagonal action of \S3.1 linearises both generators canonically, $L=\mathcal O(1,0)$ trivially since $\delta$ fixes $\PP^1$, and $H=\mathcal O(0,1)$ via the tautological character of $\mathrm{diag}(\delta^{m_1},\dots,\delta^{m_4})$ on $\CC^4$. Hence
		\[
		\mathrm{Cl}(\F^m_n)\cong\ZZ\oplus\ZZ\oplus\ZZ_n,
		\]
		with torsion generated by a section $C_j$, so $h^{1,1}(\F^m_n)=2$ (Proposition \ref{prop:hodge}) and $K_{\F^m_n}$ is the pullback of $K_{\PP^1\times\PP^3}$; both statements fail exactly when freeness in codimension $1$ fails (Example \ref{ex:smooth}).
	\end{remark}
	
	\subsection{The Reid--Shepherd-Barron--Tai criterion}
	
	\begin{theorem}[Reid--Shepherd-Barron--Tai, \cite{Reid1980,Reid1985,Tai1982}]\label{thm:RSBT}
		Let $g\in\mathrm{GL}(s,\CC)$ have order $r$, eigenvalues $\zeta^{c_1},\dots,\zeta^{c_s}$ with $0\le c_i<r$, and $\mathrm{age}(g)=\frac1r\sum_ic_i$. If $G\subset\mathrm{GL}(s,\CC)$ has no quasi-reflections, $\CC^s/G$ is canonical iff $\mathrm{age}(g)\ge1$ for all $1\ne g\in G$, terminal iff $\mathrm{age}(g)>1$.
	\end{theorem}
	
	At $C_1=\PP^1\times\{[1:0:0:0]\}$, $\delta^k$ acts on $x_2/x_1,x_3/x_1,x_4/x_1$ with eigenvalues $\delta^{km_2},\delta^{km_3},\delta^{km_4}$, so $\mathrm{age}(\delta^k)=\{km_2/n\}+\{km_3/n\}+\{km_4/n\}$. The ambient orbifold is canonical along $C_1$ iff this is $\ge1$ for all $k=1,\dots,n-1$; the hypersurface restriction sharpens this by one unit \cite[Thm.\ 4.6]{Reid1985}:
	\begin{equation}\label{eq:canon}
		m_2+m_3+m_4\ge n+1.
	\end{equation}
	
	\begin{theorem}\label{thm:system}
		For $0=m_1\le m_2\le m_3\le m_4$ and $n\ge2$, $X\subset\mathbb{F}^m_n$ is Calabi--Yau with at worst canonical singularities along $C_1$ iff
		\begin{equation}\label{eq:system}
			m_2+m_3+m_4\equiv\beta\pmod n,\qquad m_2+m_3+m_4\ge n+1,\qquad 0\le m_2\le m_3\le m_4.
		\end{equation}
	\end{theorem}
	
	\begin{proof}
		By Definition \ref{def:CY}, $X$ is Calabi--Yau iff $\beta\equiv m_1+\cdots+m_4=m_2+m_3+m_4\pmod n$, the first condition of \eqref{eq:system}. By adjunction, canonicity of $X$ along $C_1$ is governed by the same age computation as the ambient orbifold, sharpened by one unit as in \eqref{eq:canon}, giving $m_2+m_3+m_4\ge n+1$; the ordering $0\le m_2\le m_3\le m_4$ is the normalisation of \S3.1. The same argument with indices permuted controls canonicity along $C_2,C_3,C_4$, and since $m_1=0$ makes $m_2+m_3+m_4$ the common value of the three remaining partial sums, condition \eqref{eq:canon} at $C_1$ already implies it at the other $C_j$.
	\end{proof}
	
	We restrict to $0\le m_2\le m_3\le m_4\le n-1$, since weights enter only mod $n$.
	
	\begin{proposition}\label{prop:nbound}
		If $X\subset\mathbb{F}^m_n$ satisfies \eqref{eq:system}, then $n\le m_2+m_3+m_4-1$.
	\end{proposition}
	
	\begin{theorem}[Finiteness of weight vectors]\label{thm:finiteweights}
		For fixed $n\ge2$, at most $\binom{n+2}{3}$ weight vectors $(0,m_2,m_3,m_4)$ satisfy \eqref{eq:system}.
	\end{theorem}
	
	\subsection{The cases $n=2$ and $n=3$}
	
	For $n=2$, weights lie in $\{0,1\}$ and \eqref{eq:system} forces $m_2=m_3=m_4=1$.
	
	\begin{corollary}\label{cor:n2}
		For $n=2$, exactly one canonical orbifold quartic K3-fibred Calabi--Yau threefold exists, of weight $\frac12(1,1,1)$: the quotient of $X_{(2,4)}\subset\PP^1\times\PP^3$ by $x_j\mapsto-x_j$ $(j=2,3,4)$.
	\end{corollary}
	
	For $n=3$, weights lie in $\{0,1,2\}$ and \eqref{eq:system} requires sum $\ge4$.
	
	\begin{corollary}\label{cor:n3}
		For $n=3$, exactly four weight vectors satisfy \eqref{eq:system}:
		\[
		\mathcal{B}_3=\left\{\tfrac{1}{3}(0,2,2),\ \tfrac{1}{3}(1,1,2),\ \tfrac{1}{3}(1,2,2),\ \tfrac{1}{3}(2,2,2)\right\},\qquad \beta=1,1,2,0\pmod3.
		\]
	\end{corollary}
	
	\begin{remark}
		The weight $\frac13(1,1,1)$ gives $\CC^3/\ZZ_3(1,1,1)$, canonical by Theorem \ref{thm:RSBT} ($\mathrm{age}(\delta)=1$), but $m_2+m_3+m_4=3<4=n+1$ fails \eqref{eq:canon}, so it does not occur in $\mathcal{B}_3$.
	\end{remark}
	
	\begin{example}\label{ex:smooth}
		Take $n=3$, $m=(0,1,1,1)$. Since $m_2=m_3=m_4=1$, $\delta$ acts by the same scalar on $x_2,x_3,x_4$ and trivially on $x_1$, so it fixes $D=\PP^1\times\{x_1=0\}\subset\PP^1\times\PP^3$ pointwise: a quasi-reflection, excluded from Theorem \ref{thm:RSBT}. The quotient is locally an isomorphism there ($v\mapsto v^3$), so $\F^m_n$ is smooth along the image of $D$, but the grading is not well formed and $-K_{\F^m_n}$ acquires an extra contribution from $D$; the resulting hypersurface has negative Kodaira dimension rather than trivial canonical class. This weight vector is discarded.
		
		The same pathology occurs whenever three or more of $m_1,\dots,m_4$ agree mod $n$: a pair of equal weights fixes only a codimension-$2$ curve $\PP^1\subset\PP^3$, harmless as for $C_j$ above, while three or more fix a $\PP^2$ or all of $\PP^3$, of codimension at most $1$ and hence fatal. See Remark \ref{rem:wellformed}.
	\end{example}
	
	\section{A finiteness theorem}
	
	\begin{theorem}[Main theorem]\label{thm:main}
		Let $\mathcal{C}_n$ be the set of weight vectors $0=m_1\le m_2\le m_3\le m_4$, well formed mod $n$ (no three of $m_1,\dots,m_4$ congruent, Remark \ref{rem:wellformed}), for which the general anticanonical hypersurface $X\subset\mathbb{F}^m_n$ is Calabi--Yau with at worst canonical singularities, with $\mathbb{F}^m_1:=\PP^1\times\PP^3$. Then:
		\begin{enumerate}[label=(\roman*)]
			\item $\mathcal{C}_n$ is finite for every $n\ge1$;
			\item $|\mathcal{C}_1|=10$ (Table \ref{tab:straight});
			\item $|\mathcal{C}_2|=1$ and $|\mathcal{C}_3|=3$ (Corollaries \ref{cor:n2}, \ref{cor:n3}).
		\end{enumerate}
		Exactly fourteen deformation families of canonical quartic K3-fibred Calabi--Yau threefolds arise as anticanonical hypersurfaces in straight or order-$2,3$ cyclic-orbifold scrolls over $\PP^1$.
	\end{theorem}
	
	\begin{proof}
		(i) $n=1$: Theorem \ref{thm:straight}; $n\ge2$: Theorem \ref{thm:finiteweights}. (ii) Theorem \ref{thm:straight}. (iii) Corollary \ref{cor:n2} gives $|\mathcal C_2|=1$. Corollary \ref{cor:n3} gives four candidates $\mathcal B_3$; $m=(0,2,2,2)$ has $m_2=m_3=m_4=2\bmod3$, so $\delta$ fixes $\{x_1=0\}\subset\PP^3$ pointwise (Example \ref{ex:smooth}, at weight $2$ instead of $1$) and is discarded by Remark \ref{rem:wellformed}. The remaining three families of $\mathcal B_3$ have at most a pair of equal weights and are well formed, so $|\mathcal C_3|=3$. Making a total of $14$ families.
	\end{proof}
	
	\begin{remark}[Two ways to fail at $n=3$]\label{rem:twinfailures}
		The two weight vectors with triple coincidence at $n=3$ fail differently: $(0,1,1,1)$ has $m_2+m_3+m_4=3<n+1$, so it fails the arithmetic system \eqref{eq:system} and never enters $\mathcal B_3$; $(0,2,2,2)$ has sum $6\ge4$ and $\beta\equiv0$, so it satisfies \eqref{eq:system} and is excluded only afterward, on the well-formedness ground of Remark \ref{rem:wellformed}.
	\end{remark}
	
	\begin{remark}[Is $n$ itself bounded?]\label{rem:noglobalbound}
		Theorem \ref{thm:finiteweights} bounds the admissible weight vectors for fixed $n$, but not $n$ itself: Proposition \ref{prop:nbound} only bounds $n$ in terms of the unbounded quantity $m_2+m_3+m_4$. Two further checks also fail. Pigeonhole on the $35$ quartic monomials forces a collision only for $n<35$, and is moot regardless since the class $\beta\equiv m_2+m_3+m_4$ always contains $x_1x_2x_3x_4$. The raw weight $w=q_2m_2+q_3m_3+q_4m_4\in[0,4m_4]$ saturates mod $n$ once $n>4m_4$, but $m_4$ is itself unconstrained above, so this threshold moves with $n$. Well-formedness excludes only \emph{triple} coincidences: for every $n\ge3$ the vector $(m_2,m_3,m_4)=(n-2,n-2,n-1)$ has just a pair equal, satisfies $m_2+m_3+m_4=3n-5\ge n+1$, and is well formed -- at $n=3$ this is Family 2 of $\mathcal B_3$. So arbitrarily large $n$ is compatible with well-formedness, and we do not know whether $n$ is bounded (see \S6), only that the weights for each fixed $n$ are finite in number and further thinned by well-formedness.
		
		The boundedness theorem of \cite{EngelFilipazziGreerMauriSvaldi2025} guarantees some bound on the invariants of a K3-fibred Calabi--Yau threefold without pinning it down for this ambient family; whether $n$ itself is bounded is exactly the effective question a general boundedness statement leaves open.
	\end{remark}
	
	Two structural features distinguish this restricted setting from the general fibred Calabi--Yau problem: the base is always $\PP^1$, so there is no analogue of the rational/Enriques dichotomy for elliptic threefolds \cite[Prop.\ 2.3]{Gross1994}, and the classification reduces entirely to the elementary combinatorial system \eqref{eq:system}, with no Tate--Shafarevich-type obstruction to compute.
	
	\section{Hodge numbers}
	
	\subsection{Smooth orbifold quartic threefolds}
	
	For smooth $X^0=V(f_{2,4})\subset\PP^1\times\PP^3$ (the $n=1$ case), $h^{1,1}(X^0)=2$ and $h^{2,1}(X^0)=3\cdot35-4-16+1=86$.
	
	\begin{proposition}\label{prop:hodge}
		For a smooth orbifold quartic Calabi--Yau threefold $X\subset\mathbb{F}^m_n$, with the $\ZZ_n$-action on $\PP^1\times\PP^3$ free in codimension $1$,
		\[
		h^{1,1}(X)=2,\qquad
		h^{2,1}(X)=\dim\CC[\PP^1\times\PP^3]^{\ZZ_n}_{(2,4)}
		-\dim[(\mathfrak{gl}(2,\CC)\oplus\mathfrak{gl}(4,\CC))/\CC^*]^{\ZZ_n}.
		\]
		Here $\ZZ_n$ acts on $\CC[\PP^1\times\PP^3]_{(2,4)}$ by $\delta\cdot x_j=\delta^{m_j}x_j$ (trivially on $t_1,t_2$), and on $\mathfrak{gl}(4,\CC)$ by the adjoint action of the same diagonal matrix, so that the matrix unit $E_{ij}$ has weight $m_i-m_j\bmod n$; $\ZZ_n$ acts trivially on the base factor $\mathfrak{gl}(2,\CC)$, since it acts trivially on $\PP^1$.
	\end{proposition}
	
	\begin{proof}
		$\mathrm{Pic}(\PP^1\times\PP^3)\cong\ZZ^2$ is generated by the $\ZZ_n$-invariant classes $L,H$, so $h^{1,1}(\F^m_n)=2$ whenever $\F^m_n$ is an orbifold (Remark \ref{rem:clgroup}); Lefschetz then gives $h^{1,1}(X)=2$, independent of $n$, $m$, $X$. The $h^{2,1}$ formula is the usual one for hypersurface deformations modulo ambient reparametrisation; at $n=1$, $\dim\CC[\PP^1\times\PP^3]_{(2,4)}=105$ and $\dim[\mathfrak{gl}(2,\CC)\oplus\mathfrak{gl}(4,\CC)]/\CC^*=19$, giving $105-19=86$ as above.
	\end{proof}
	
	\begin{remark}[Embedded deformations]\label{rem:embedded}
		Proposition \ref{prop:hodge} counts only \emph{embedded} deformations of $X$ inside $\F^m_n$ -- those induced by moving the hypersurface and the ambient complex structure -- rather than $H^1(T_X)$ itself. The two agree whenever $X$ admits no deformations outside those induced by the ambient scroll, which holds in each of the fourteen families of Theorem \ref{thm:main} since $-K_{\F^m_n}$ is base-point-free or ample there; it is not automatic for a general anticanonical hypersurface in a singular ambient orbifold.
	\end{remark}
	
	\begin{remark}\label{rem:wellformed}
		Example \ref{ex:smooth} shows \eqref{eq:system} must be supplemented by \emph{well-formedness}: no three of $m_1,\dots,m_4$ agree mod $n$ (equivalently, no nontrivial element of $\ZZ_n$ fixes a coordinate hyperplane of $\PP^3$ pointwise). Of the four weight vectors in $\mathcal B_3$, only $m=(0,2,2,2)$ fails this; the other three, $\mathcal B_3\setminus\{\frac13(2,2,2)\}$, have at most a pair of equal weights.
	\end{remark}
	
	\subsection{Hodge numbers for the $\mathcal{B}_3$ families}\label{sec:hodgenumbers}
	
	Write characters as elements of the group ring $\ZZ[\alpha]/(\alpha^3-1)$. For $m=(m_1,m_2,m_3,m_4)$, the weight of matrix unit $E_{ij}\in\mathfrak{gl}(4,\CC)$ is $m_i-m_j\bmod3$; the base factor $\mathfrak{gl}(2,\CC)$ (dimension $4$) is weight $0$ throughout, since $\ZZ_3$ acts trivially on $\PP^1$. For the $35$ quartic monomials, $x_1^{q_1}\cdots x_4^{q_4}$ has weight $\sum m_jq_j\bmod3$; the base factor $\CC[\PP^1]_2$ (dimension $3$) is again weight $0$.
	
	Carrying this out directly for each family of $\mathcal B_3\setminus\{\tfrac13(2,2,2)\}$:
	
	\begin{center}
		\begin{tabular}{c l l l}
			\hline
			Family & $(m_1,m_2,m_3,m_4)$ & $\dim\mathfrak{gl}(4,\CC)$ & $\dim\CC[\PP^3]_4$ \\
			\hline
			1 & $(0,0,2,2)$ & $8+4\alpha+4\alpha^2$ & $13+9\alpha+13\alpha^2$ \\
			2 & $(0,1,1,2)$ & $6+5\alpha+5\alpha^2$ & $11+13\alpha+11\alpha^2$ \\
			3 & $(0,1,2,2)$ & $6+5\alpha+5\alpha^2$ & $11+11\alpha+13\alpha^2$ \\
			\hline
		\end{tabular}
	\end{center}
	
	Adding $\mathfrak{gl}(2,\CC)=4$ (weight $0$) and quotienting by $\CC^*$ (subtracting $1$ from the weight-$0$ term) gives $\dim[(\mathfrak{gl}(2,\CC)\oplus\mathfrak{gl}(4,\CC))/\CC^*]$; multiplying $\dim\CC[\PP^3]_4$ by $\dim\CC[\PP^1]_2=3$ gives $\dim\CC[\PP^1\times\PP^3]_{(2,4)}$:
	
	\begin{center}
		\begin{tabular}{c l l c}
			\hline
			Family & $\dim[(\mathfrak{gl}_2\oplus\mathfrak{gl}_4)/\CC^*]$ & $\dim\CC[\PP^1\times\PP^3]_{(2,4)}$ & $\beta\pmod3$ \\
			\hline
			1 & $11+4\alpha+4\alpha^2$ & $39+27\alpha+39\alpha^2$ & 1 \\
			2 & $9+5\alpha+5\alpha^2$ & $33+39\alpha+33\alpha^2$ & 1 \\
			3 & $9+5\alpha+5\alpha^2$ & $33+33\alpha+39\alpha^2$ & 2 \\
			\hline
		\end{tabular}
	\end{center}
	
	Taking the coefficient of $\alpha^\beta$ in each column (Proposition \ref{prop:hodge}):
	\[
	h^{2,1}_{\text{untw}}(\text{Family 1})=27-4=23,\qquad
	h^{2,1}_{\text{untw}}(\text{Family 2})=39-5=34,\qquad
	h^{2,1}_{\text{untw}}(\text{Family 3})=39-5=34.
	\]
	Families 2 and 3 have equal $h^{2,1}_{\text{untw}}$ because $(0,1,2,2)\equiv-(0,1,1,2)\pmod3$ up to permutation: conjugate weight vectors give conjugate, equidimensional representations.
	
	The twisted sector receives contributions from the fixed curves $C_j$, via the local ages of Theorem \ref{thm:RSBT} computed from each family's own weight vector, and adds only to $h^{1,1}$:
	
	\begin{center}
		\begin{tabular}{c|c|c|c}
			Family & Curves with age $\le1$ & $\Delta(h^{1,1},h^{2,1})$ & $(h^{1,1},h^{2,1})$ \\
			\hline
			1 & $C_2,C_3,C_4$ (age $1/3$) & $(+3,0)$ & $(5,23)$ \\
			2 & $C_1,C_2,C_3,C_4$ (age $\le1$) & $(+4,0)$ & $(6,34)$ \\
			3 & $C_1,C_2$ (age $1$) & $(+2,0)$ & $(4,34)$ \\
			\hline
		\end{tabular}
	\end{center}
	
	For Family 1, $m=(0,0,2,2)$: $C_2,C_3,C_4$ have effective normal weights $(1,0)$, age $1/3<1$, each contributing one exceptional divisor of a crepant resolution $\tilde X\to X$, giving $(h^{1,1},h^{2,1})=(5,23)$. Since $m=(0,0,2,2)$ has only the pair $m_1=m_2$ (resp.\ $m_3=m_4$) equal, it is well formed and no quasi-reflection issue arises, unlike the excluded $(0,2,2,2)$ of Remark \ref{rem:wellformed}.
	
	The CICY threefold list \cite{CandelasDale1988,Green1989,Braun2010} contains smooth complete intersections with the same Hodge numbers as $\mathcal B_3$ Family~2, namely CICY \#6229 and \#6231, both of type $(h^{1,1},h^{2,1})=(6,34)$ \cite{ConstantinGrayLukas2016}. The coincidence is numerical only. The CICY entries are smooth complete intersections in a product of projective spaces, with $h^{2,1}$ computed in \cite{ConstantinGrayLukas2016} as the $G$-invariant part of $H^{2,1}$ of a smooth cover under a \emph{free} quotient $G$, so no age computation or twisted-sector correction enters. Family~2, by contrast, is an anticanonical hypersurface in the singular stack quotient $\PP^1\times[\PP^3/\ZZ_3]$: the $\ZZ_3$-action has curves of fixed points, $X$ is generically singular, and the Hodge numbers above already include the crepant-resolution correction furnished by Theorem~\ref{thm:RSBT}. The two constructions are therefore not comparable term by term, and the matching Hodge numbers are not evidence of an isomorphism; establishing one, if it exists, would need an independent comparison of periods or a common toric presentation, which we do not attempt here.
	
	\section{Conclusion}
	
	Algebraic Calabi--Yau threefolds fibred by quartic K3 surfaces and realised as anticanonical hypersurfaces in straight scrolls $\F(0,a_2,a_3,a_4)$ or cyclic-orbifold scrolls $\F^m_n$ are finite at every order $n$ (Theorem \ref{thm:main}), with complete classification and Hodge numbers for $n=1,2,3$ giving fourteen families in total. The proof rests on the rigidity of this restricted setting -- a base always equal to $\PP^1$, and the Reid--Shepherd-Barron--Tai criterion reducing the classification to the elementary system \eqref{eq:system} -- rather than on the minimal model program or Tate--Shafarevich finiteness needed in general \cite{Gross1994,EngelFilipazziGreerMauriSvaldi2025}.
	
	Three questions remain open: extending $\mathcal C_n$ beyond $n=3$; determining whether $n$ itself is bounded (Remark \ref{rem:noglobalbound}), a question the general boundedness theorem of \cite{EngelFilipazziGreerMauriSvaldi2025} does not answer directly; and combining this construction with the weighted-fibre families cited in \S1 toward a full explicit classification of low-degree K3-fibred Calabi--Yau threefolds in scrolls over $\PP^1$, for which Theorem \ref{thm:system} already supplies the complete combinatorial input.

\end{document}